\documentclass[a4paper]{article}

\usepackage{
amsmath,
amsthm,
amscd,
amssymb,
wasysym
}
\usepackage{fancyvrb}

\DefineVerbatimEnvironment{Verbatim}{Verbatim}{fontsize=\footnotesize}

\usepackage{pgfplotstable}
\usepackage{comment}
\usepackage{tikz}
\usetikzlibrary{positioning,arrows,calc}
\usepackage{xspace}

\usepackage{graphicx}

\usepackage{url}
\usepackage{subcaption}

\theoremstyle{plain}
\newtheorem{lemma}{Lemma}
\newtheorem{definition}{Definition}
\newtheorem{corollary}{Corollary}
\newtheorem{proposition}{Proposition}
\newtheorem{theorem}{Theorem}
\newtheorem{conjecture}{Conjecture}
\newtheorem{question}{Question}

\newtheoremstyle{derp}
{3pt}
{3pt}
{}
{}
{\upshape}
{:}
{.5em}
{}
\theoremstyle{derp}
\newtheorem{example}{Example}

\newcommand{\R}{\mathbb{R}}

\newcommand{\Z}{\mathbb{Z}}

\newcommand{\N}{\mathbb{N}}

\newcommand\xqed[1]{%
  \leavevmode\unskip\penalty9999 \hbox{}\nobreak\hfill
  \quad\hbox{#1}}
\newcommand\qee{\xqed{$\fullmoon$}}

\newcommand{\Fin}{\mathrm{Fin}}
\newcommand{\FinG}{\mathrm{FinGoE}}

\newcommand{\one}[1]{#1}
\newcommand{\two}[1]{\tilde{#1}}

\newcommand{\dd}[2]{\begin{matrix} #1 \\ #2 \end{matrix}}

\newcommand{\lang}[1]{\mathcal{L}(#1)}

\newcommand{\pad}{\mathrm{pad}_0}
\newcommand{\conf}{\mathrm{conf}_0}

\makeatletter
\tikzset{
    zero color/.initial=white,
    zero color/.get=\zerocol,
    zero color/.store in=\zerocol,
    one color/.initial=red,
    one color/.get=\onecol,
    one color/.store in=\onecol,
    two color/.initial=blue,
    two color/.get=\twocol,
    two color/.store in=\twocol,
    cell wd/.initial=1,
    cell wd/.get=\cellwd,
    cell wd/.store in=\cellwd,
    cell ht/.initial=1,
    cell ht/.get=\cellht,
    cell ht/.store in=\cellht,
}

\newcommand{\drawgrid}[2][]{
\medskip
\begin{tikzpicture}[#1]
  \pgfplotstablegetrowsof{#2} 
  \pgfmathtruncatemacro{\totrow}{\pgfplotsretval}
  \pgfplotstablegetcolsof{#2} 
  \pgfmathtruncatemacro{\totcol}{\pgfplotsretval}
  
  \pgfplotstableforeachcolumn#2\as\col{
    \pgfplotstableforeachcolumnelement{\col}\of#2\as\colcnt{%
      \ifnum\colcnt=0
        \fill[\zerocol]($ -\pgfplotstablerow*(0,\cellht) + \col*(\cellwd,0) $) rectangle+(\cellwd,\cellht);
      \fi
      \ifnum\colcnt=1
        \fill[\onecol]($ -\pgfplotstablerow*(0,\cellht) + \col*(\cellwd,0) $) rectangle+(\cellwd,\cellht);
      \fi
      \ifnum\colcnt=2
        \fill[\twocol]($ -\pgfplotstablerow*(0,\cellht) + \col*(\cellwd,0) $) rectangle+(\cellwd,\cellht);
      \fi
    }
  }
\end{tikzpicture}
\medskip
}

\title{Gardens of Eden in the Game of Life}

\author{
  Ville Salo
  \footnote{Author supported by Academy of Finland grant 2608073211.}
  \ \ \ \ \ \ and\ \ \ \ \
  Ilkka T\"orm\"a
  \footnote{Author supported by Academy of Finland grant 295095.}
  \\
  Department of Mathematics and Statistics \\
  University of Turku, Finland \\
  \{\texttt{vosalo}, \texttt{iatorm}\}\texttt{@utu.fi}
}

\begin{document}
\maketitle

\begin{abstract}
We prove that in the Game of Life,
if the thickness-four zero-padding of a rectangular pattern is not an orphan, then the corresponding finite-support configuration is not a Garden of Eden, and that
the preimage of every finite-support configuration has dense semilinear configurations. 
 In particular finite-support Gardens of Eden are in co-NP.
\end{abstract}

\section{Introduction}

The Game of Life is a two-dimensional cellular automaton defined by John Conway in 1970~\cite{Ga70}.
It consists of an infinite grid of cells, each of which is either dead or alive, and a discrete time dynamics defined by a simple local rule: a live cell survives if it has 2 or 3 live neighbors, and a dead cell becomes live if it has exactly 3 live neighbors.
We study the \emph{Gardens of Eden} of the Game of Life, that is, configurations that do not have a predecessor in the dynamics.
The following theorem is our main result.

\begin{theorem}
The set of finite-support Gardens of Eden for the Game of Life is in co-NP under the encoding where the input is a rectangular pattern containing all live cells.
\end{theorem}

Equivalently, the input can specify the values in any area of polynomial diameter, since we can always extend the domain to a rectangular one. 
A configuration has \emph{finite support} if it contains a finite number of live cells. Finite-support Gardens of Eden should be contrasted with the \emph{orphans}, which are patterns of finite domain that do not appear in the image subshift of the Game of Life. The set of orphans is in co-NP for trivial reasons.

This result is a corollary of either of the following theorems.
In the first case, we also need some quantitative details of the proof.

\begin{theorem}
\label{thm:dense}
Let $g$ be the Game of Life cellular automaton. If $y$ is a configuration with finite support, then semilinear configurations are dense in $g^{-1}(y)$.
\end{theorem}

\begin{theorem}
\label{thm:padded}
Let $g$ be the Game of Life cellular automaton, and let $P$ be a rectangular pattern. If the thickness-$4$ zero-padding of $P$ admits a preimage $Q$, then the finite-support configuration $y$ corresponding to $P$ has a preimage.
\end{theorem}

Theorem~\ref{thm:dense} is not true if ``semilinear'' is replaced with ``finite-support'' or ``co-finite-support''.
In Theorem~\ref{thm:padded} we do not know whether the optimal constant is $4$, but it is at least $1$. We state a more precise result in Theorem~\ref{thm:Main}, which additionally guarantees that we need not modify the preimage of the padding of $P$ in the part that maps to $P$, to obtain a preimage for the finite-support configuration $y$. For this stronger fact, the thickness $4$ is optimal, and no convex compact shape other than a rectangle can be used.

The seed of the proofs of the above theorems is the following property of the Game of Life: the set of rectangular patterns of height $2$ that can be extended to a preimage of the all-$0$ configuration (the \emph{trace} of the subshift of finite type $g^{-1}(0^{\Z^2})$) is a regular language, and all such patterns can be extended so that the preimage is vertically $3$-periodic everywhere except within distance $3$ of the pattern.
The essential idea in the proofs of both of the theorems above is to apply this result on all sides of a rectangle containing the desired image, to change the preimage to a ``better'' one.
Once the regularity of the trace language has been established, the periodic extension property is decidable using automata theory, and we decide both by computer. The proof did not seem to be within reach of any standard library we attempted to use, so we used pure Python instead, and our program is included for verification in Appendix~\ref{sec:Program}. We expect that some efficient enough standard libraries could solve this problem out of the box.

The technical lemmas of this paper are stated for general two-dimensional cellular automata, and the computed-assisted part can be adapted to a general cellular automaton with little work by modifying the attached program. One obtains analogues of the above theorems for any cellular automaton where the preimage of the all-zero configuration has one-sided stable/periodizable traces.
Our methods to not generalize to higher-dimensional cellular automata.

This paper arose from the answer of Oscar Cunningham to the question \cite{Zy19} on MathOverflow, which in turn arose from the thread \cite{OrphanvsGoE} on the ConwayLife forum. Cunningham asked whether the finite-support Gardens of Eden of the Game of Life cellular automaton form a decidable set. It follows from compactness that they are a semidecidable set for any cellular automaton (as every Garden of Eden contains an orphan), but there is no general-purpose semidecision algorithm for the other direction, in the sense that there exist two-dimensional cellular automata whose finite-support Gardens of Eden are undecidable. Our results of course solve the problem for the Game of Life.

\begin{corollary}
Finite-support Gardens of Eden for the Game of Life are a decidable set, under any natural encoding.
\end{corollary}

The methods of this article are special cases of a more general technique we are studying with Pierre Guillon. Trace methods are a common tool in multidimensional symbolic dynamics and cellular automata theory \cite{Pa12,PaSc15,Le06}. See \cite{Ga83,Ad10,Re02} for discussion of various aspects and generalizations of the Game of Life.

\section{Definitions}

We include $0 \in \N$, and unless otherwise noted, all intervals are discrete: $[a, b] = [a, b] \cap \Z$.

For $A$ a finite set with a special element $0 \in A$ called the \emph{zero}, a $d$-dimensional \emph{configuration} over $A$ is an element $x \in A^{\Z^d}$.
The elements of $\Z^d$ are called \emph{cells}, and the value of a cell $\vec v$ in $x$ is denoted $x_{\vec v}$.
The set $A^{\Z^d}$ of all configurations is the \emph{$d$-dimensional full shift}, and we give it the prodiscrete (Cantor) topology.
The additive group $\Z^d$ acts on $A^{\Z^d}$ by shifts: $\sigma^{\vec v}(x)_{\vec w} = x_{\vec v + \vec w}$.
A configuration $x$ is \emph{finite-support} if $x_{\vec v} = 0$ for all but a finite number of cells $\vec v \in \Z^d$, and we denote by $\Fin(A)$ the set of finite-support configurations. If $X \subset A^\Z$ is a one-dimensional subshift, write $\lang{X} = \{w \in A^* \;|\; \exists x \in X: x|_{[0,|w|-1]} = w\}$ for its \emph{language}, the set of (finite-length) words that can be extended to infinite configurations in $X$.

A \emph{pattern} is an element $P \in A^D$, where $D = D(P) \subset \Z^d$ is the \emph{domain} of $P$, and we say $P$ is a \emph{(finite) pattern} if it has a finite domain. If $P \in A^D$ then $\sigma^{\vec v}(P) \in A^{D - \vec v}$ is defined by $\sigma^{\vec v}(P)_{\vec u} = P_{\vec u + \vec v}$. If $P \in A^D$ is a pattern, write $\pad^c(P) \in A^{D + [-c, c]^d}$ for the pattern together with a zero-padding of thickness $c$ on all sides, and $\conf(P)$ is the finite-support configuration corresponding to $P$ defined by $x_{\vec v} = P_{\vec v}$ for $\vec v \in D$, and $x_{\vec v} = 0$ otherwise. 
If $P \in A^D$ and $Q \in A^E$ are two patterns, write $P \sqsubset Q$ if there exists $\vec v \in \Z^d$ with $D + \vec v \subset E$ and $\sigma^{\vec v}(Q)|_D = P$.
In particular, if $E = \Z^d$ then $Q = x$ is a configuration, and $P \sqsubset x$ means that $P$ occurs somewhere in $x$.

A \emph{subshift} is a topologically closed and shift-invariant set $X \subset A^{\Z^d}$.
It follows from compactness that every subshift is defined by some set of \emph{forbidden patterns} $F$ as
\[
  X = X_F = \{ x \in A^{\Z^d} \;|\; \forall P \in F : P \not\sqsubset x \},
\]
If $F$ can be chosen finite, then $X$ is a \emph{shift of finite type}, or \emph{SFT} for short.

A \emph{cellular automaton} or \emph{CA} is a function $f : A^{\Z^d} \to A^{\Z^d}$ that commutes with shifts and is continuous for the prodiscrete topology of $A^{\Z^d}$.
By the Curtis-Hedlund-Lyndon theorem~\cite{He69}, every CA is defined by a finite \emph{neighborhood} $N \subset \Z^d$ and a \emph{local rule} $F : A^N \to A$ by $f(x)_{\vec v} = F(\sigma^{\vec v}(x)|_N)$. Usually we have $N = [-r, r]^d$ for some \emph{radius} $r$ and think of the local rule $F : A^N \to A$ as part of the structure of the CA. If $P \in A^D$ is a finite pattern, we can apply $f$ to $P$ by defining $E = \{ \vec v \in D \;|\; \vec v + N \subset D \}$ and $f(P) = Q \in A^E$ where $Q_{\vec v} = F(\sigma^{\vec v}(Q)|_N)$.
A \emph{Garden of Eden} for $f$ is a configuration $x \in A^{\Z^d}$ such that $f^{-1}(x) = \emptyset$; equivalently, $x \in A^{\Z^d} \setminus f(A^{\Z^d})$. The \emph{finite-support Gardens of Eden} of $f$ are
\[ \FinG(f) = \Fin(A) \setminus f(A^{\Z^d}). \]
An \emph{orphan} is a finite pattern $P$ such that $P \not\sqsubset f(x)$ for all $x \in A^{\Z^d}$.

The images of SFTs under cellular automata are \emph{sofic shifts}.
A one-dimensional subshift is sofic iff its language is regular, iff it can be defined by a regular language of forbidden words.

The \emph{Game of Life} is the two-dimensional cellular automaton $g : A^{\Z^2} \to A^{\Z^2}$ over $A = \{0,1\}$ defined by
\[ g(x)_{(a,b)} = \left\{\begin{array}{ll}
1 & \mbox{if } x_{(a, b)} = 0 \mbox{ and } \sum x|_{(a, b) + B} = 3 \\
1 & \mbox{if } x_{(a, b)} = 1 \mbox{ and } \sum x|_{(a, b) + B} \in \{3, 4\} \\
0 & \mbox{otherwise},
\end{array}\right. \]
where $B = \{-1,0,1\}^2$.
It has radius $1$.

A configuration $x \in A^{\Z^d}$ is \emph{semilinear} if for every symbol $a \in A$, the set $\{\vec v \in \Z^d \;|\; x_{\vec v} = a\}$ is \emph{semilinear}, meaning it is a finite union of \emph{linear} sets. A \emph{linear} set is a set of the form $\vec v + \langle \vec v_1, \vec v_2, ..., \vec v_k \rangle$, where $\vec v, \vec v_i \in \Z^d$, and $\langle V \rangle$ denotes the monoid generated by $V$. It is well-known that a cellular automaton image of a semilinear configuration is effectively semilinear. This is a very robust class of sets, see \cite{GiSp66} for characterizations.

For the purpose of complexity theory and computability theory, we fix a bijection between $B^*$ and $\Fin(A)$ for a fixed alphabet $B$, so that $\FinG(f)$ can be seen as a language. For $A = \{0,1,...,|A|-1\}$, $|A| \geq 2$ and $d = 2$, we fix $B = A$ and use the encoding where $w = 0^M 1^N 0 u$ with $|u| = (2M+1)(2N+1)$ represents the finite-support configuration $x \in A^{\Z^2}$ where $x_{(a, b)} = 0$ if $(a, b) \notin [-M, M] \times [-N, N]$ and $x_{(a,b)} = u_{(2M+1)a + b}$ otherwise. A similar encoding is used for finite patterns with domains of the form $[-M,M] \times [-N, N]$.

For the remainder of this article, we fix $d = 2$. Up and north are synonyms (and refer to the vector $(0,1)$), and similarly left, right and down are synonymous with west, east and south, respectively.

We need some basic facts and definitions of symbolic dynamics, automata theory and complexity theory, some standard references are \cite{LiMa95,PePi04,HoMoUl06,ArBa09}.
In particular, all Boolean operations on regular languages, the concatenation operation $KL = \{uv \;|\; u \in K, v \in L\}$, and the equality of two given regular languages, are computable.

\section{Traces}

\subsection{Periodizable traces}

A pattern $P \in A^D$ is \emph{height-$n$} if $D \subset \Z \times [0,n-1]$. We define \emph{width-$n$} similarly, and $P$ is \emph{size-$n$} if it is width-$n$ and height-$n$. For a set of finite patterns $F$, define $X^\N_F = \{x \in A^{\Z \times \N} \;|\; \forall P \in F : P \not\sqsubset x \}$ and $X^n_F = \{x \in A^{\Z \times [0,n-1]} \;|\; \forall P \in F : P \not\sqsubset x \}$. 
These are the upper half-planes and horizontal stripes of height $n$ where patterns of $F$ do not occur.

\begin{definition}
Let $F$ be a set of forbidden patterns. Write
\[ T_n(F) = \{x|_{\Z \times [0,n-1]} \;|\; x \in X^\N_F\} \subset A^{\Z \times [0,n-1]} \]
for the \emph{one-sided trace} of $X_F$. Define $\pi_n : \bigcup_{[0,n-1] \subset I \subset \Z} A^{\Z \times I} \to A^{\Z \times [0,n-1]}$ by $\pi_n(x) = x|_{\Z \times [0, n-1]}$. Define
\[ S_{n,\ell}(F) = \pi_n(\one{X}_F^{n+\ell}). \]
For $k \geq 0$ and $p \geq 1$ define
\[ P_{n,k,p}(F) = \{x|_{\Z \times [0,n-1]} \;|\; x \in X^\N_F, \forall a \in \Z, b \geq n+k: x_{(a,b)} = x_{(a, b+p)} \}. \]
\end{definition}

The trace $T_n(F)$ is the set of those height-$n$ stripes that can be extended to the entire upper half-plane without introducing a pattern of $F$, while the stripes of $S_{n,\ell}(F)$ can be extended by $\ell$ additional rows, and the stripes of $P_{n,k,p}(F)$ can be extended into upper half-planes that are vertically $p$-periodic after $k$ rows.
The inclusions $P_{n,k,p}(F) \subset T_n(F) \subset S_{n,\ell}(F)$ always hold, and $T_n(F) = \bigcap_{\ell \in \N} S_{n,\ell}(F)$.
We often identify $A^{\Z \times [0,n-1]}$ with $(A^{[0,n-1]})^\Z$, so that the traces can be seen as infinite sequences of finite patterns of shape $1 \times n$.
We could also define two-sided traces in the analogous way, but we have omitted them from this article since they are not needed for our results, computing them is more resource-intensive, and the relationships between the relevant properties of one-sided and two-sided traces is not entirely trivial.

Write $\circlearrowright$ for the operation that rotates a pattern, configuration, or every element of a set of such, 90 degrees clockwise around the origin. 

\begin{definition}
Let $F$ be a set of forbidden patterns of size $n+1$. We say $F$ has \emph{one-sided periodizable traces} if there exist $k, p \in \N$ with $p \geq 1$ such that
\begin{equation} T_n({\circlearrowright^i}(F)) \subset \one{P}_{n,k,p}({\circlearrowright^i}(F))
\label{eq:noice} \end{equation}
holds for all $i \in \{0,1,2,3\}$.
\end{definition}

For $i = 0$ this property means that if an arbitrary row of height $n$ (too thin to contain forbidden patterns) can be extended upward into a half-plane that contains no forbidden patterns, then it can also be extended by $k+p$ rows so that the last $p$ rows can be repeated periodically. For other $i$, we can interpret this as a similar property for extensions to the left, downward and right. If $X \subset A^{\Z^2}$ is an SFT where a set of forbidden patterns $F$ of height at most $n+1$ has been fixed, we sometimes say $X$ has one-sided periodizable traces if $F$ does.

\begin{lemma}
\label{lem:PeriodizableImplies}
Let $f : A^{\Z^2} \to A^{\Z^2}$ be a cellular automaton with neighborhood $[-r,r]^2$. 
Let $F$ be the patterns $P \in A^{[-r, r]^2}$ that map to a nonzero symbol in the local rule of $f$. 
If $F$ has one-sided periodizable traces, then semilinear configurations are dense in $f^{-1}(y)$ for every $y \in \Fin(A)$, and $\FinG(f)$ is in co-NP.
\end{lemma}

\begin{proof}
Denote $n = 2 r$, and let  $k, p \in \N$ be given by the assumption of one-sided periodizable traces. We may assume $p \geq n$. Let $y \in \Fin(A)$ be arbitary. We prove that semilinear configurations are dense in $f^{-1}(y)$. If $y$ has no preimage, we are done; otherwise, take an arbitrary preimage $f(x) = y$, and let $N > p$ be such that the support of $y$ is contained in $D = [-N,N]^2$. We construct a semilinear preimage of $y$ that agrees with $x$ inside $R = [-N-r, N+r]^2$.

Consider the upper half-plane containing the top $r$ rows of $D$ and everything above them.
We have $\sigma^{(0, N-r+1)}(x)|_{\Z \times [0, \infty)} = z \in X^\N_F$.
Then the restricton to the bottom $2r$ rows of this half-plane satisfies $\sigma^{(0, N-r+1)}(x)|_{\Z \times [0, 2r-1]} = \pi_{2r}(z) \in T_{2r}(F) = P_{2r,k,p}(F)$, where the last equality is given by \eqref{eq:noice} with $i = 0$.
This means we can replace the contents of an upper half-plane in $x$ with a vertically periodic pattern, i.e.\ there exists $x^1 \in f^{-1}(y)$ such that $x^1|_{\Z \times (-\infty, N+r]} = x|_{\Z \times (-\infty, N+r]}$ and $x^1_{(a, b)} = x^1_{(a, b+p)}$ for all $a \in \Z$ and $b > N+r+k$.
We can apply the exact same argument to the configuration $x^1$ below the rectangle $[-N, N]^2$, using \eqref{eq:noice} for $i = 2$, giving us $x^2 \in f^{-1}(y)$ such that
$x^2|_{\Z \times [-N-r, N+r]} = x|_{\Z \times [-N-r, N+r]}$ and $x^2_{(a, b)} = x^2_{(a, b+p)}$ whenever $a \in \Z$ and $\min(|b| , |b+p|) > N+r+k$.

Next, we apply the same argument on the left and right borders of $D$ in $x^2$, using \eqref{eq:noice} with $i = 1, 3$ in either order. This gives us a preimage $x^3 \in f^{-1}(y)$ such that $x^3|_{[-N-r, N+r]^2} = x|_{[-N-r, N+r]^2}$ and, defining $A_{\mathrm{north}} = [-N-r, N+r] \times [N+r+k+1, \infty)$, $A_{\mathrm{south}} = [-N-r, N+r] \times (-\infty, -N-r-k-1]$, $A_{\mathrm{west}} = (-\infty, -N-r-k-1] \times \Z$, $A_{\mathrm{east}} = [N+r+k+1, \infty) \times \Z$, we have $x^3_{(a, b)} = x^3_{(a, b+p)}$ whenever $(a, b), (a,b+p) \in A_{\mathrm{north}} \cup A_{\mathrm{south}}$, and
$x^3_{(a, b)} = x^3_{(a+p, b)}$ whenever $(a, b), (a+p, b) \in A_{\mathrm{west}} \cup A_{\mathrm{east}}$.

The ``aperiodic region'' $B = \Z^2 \setminus (A_{\mathrm{north}} \cup A_{\mathrm{south}} \cup  A_{\mathrm{west}} \cup A_{\mathrm{east}})$ that is left after this process is infinite, so we need a final periodization step to obtain semilinearity. Since $A_{\mathrm{north}}$ is $(0,p)$-periodic, and rows of $A_{\mathrm{west}}$ and $A_{\mathrm{east}}$ are $(p,0)$-periodic, the number of distinct height-$n$ stripes $\sigma^{(0,h)}(x)|_{\Z \times [0, n-1]}$ for $h > N+r+k$ is bounded by $q = p |A|^{2 n (k + p)}$, where the factor $p$ comes from the phase of the period in $A_{\mathrm{north}}$ and the term $|A|^{2 n (k + p)}$ comes from the two patterns of shape $n \times (k+p)$ just to the the west and east sides of $A_{\mathrm{north}}$, which include the width-$k$ aperiodic regions and the repeating parts of $A_{\mathrm{west}}$ and $A_{\mathrm{east}}$.

Thus the stripes defined by some $N+r+k < h_1 < h_2 \leq N+r+k + q + 1$ are equal, and since we assumed $p \geq 2r$, we can form a new configuration $x^4 \in f^{-1}(y)$ by repeating the part between there stripes in the upper half-plane $[h_1, \infty)$.
Then $x^4$ agrees with $x^3$ on $\Z \times (-\infty,N+r]$ and is vertically $p'$-periodic in $[h_1, \infty)$ for some $p' \leq q$.
We apply the same argument to the south half-plane of $x^4$, obtaining a configuration $x^5 \in f^{-1}(y)$ with $x|_{[-N-r, N+r]^2} = x^5|_{[-N-r, N+r]^2}$ that has horizontal period $p$ outside $[-N-r-k, N+r+k] \times \Z$ and vertical period $q!$ outside $\Z \times [-q, q]$.
In particular, $x^5$ is semilinear.

The co-NP claim follows from the quantitative statements above about the semilinear preimages. A semilinear preimage with the periodicity properties of $x^5$ can be summarized as a polynomial-size certificate, by giving the restriction
\[ x^5|_{[-N-r-k-p, N+r+k+p] \times [-N-q-q!, N+q+q!]}. \]
One can check in polynomial time that continuing the periods does not give a forbidden pattern for $X$.
Note that while $q!$ grows very fast as a function of $|A|$, $p$ and $k$, it is a constant for any fixed CA that satisfies the assumptions.
This proves that the complement of $\FinG(f)$ is in NP, as claimed.
\end{proof}

We include the statement that semilinear preimages imply co-NP, because this gives a very natural certificate -- an actual preimage. However, the details of working computationally with semilinear configurations, while standard, are not entirely trivial, and are omitted in the above proof. For the co-NP certificate obtained from Lemma~\ref{lem:StableImplies} the verification algorithm is much more obvious.

\subsection{Stable traces}

There is no general method of computing the one-sided traces $T_n(F)$ exactly.
On the other hand, the languages $\lang{S_{n,\ell}(F)}$ of the approximate traces are regular and easy (though often resource-intensive) to compute.
If the sequence $(S_{n,\ell}(F))_{\ell \in \N}$ stabilizes after finitely many steps, then $\lang{T_n(F)}$ is also a regular language and we can analyze it using finite automata theory.

\begin{definition}
If $S_{n,\ell}({\circlearrowright^i}(F))) = S_{n,\ell+1}({\circlearrowright^i}(F))$ for some $\ell \in \N$ and all $i \in \{0,1,2,3\}$ and $F$ has size at most $n+1$, then we say $F$ has \emph{one-sided stable traces}.
\end{definition}

%
%

The following is shown by induction, by extending a configuration legally, row by row, obtaining a valid half-plane in the limit.

\begin{lemma}
\label{lem:StableImpliesActualTrace}
If $F$ has height at most $n+1$ and $S_{n,\ell}(F)) = S_{n,\ell+1}(F)$, then $S_{n,\ell}(F) = S_{n,\ell+m}(F) = T_n(F)$ for all $m \geq 0$.
\end{lemma}

We recall a basic symbolic dynamics lemma, a version of the pumping lemma.

\begin{lemma}
\label{lem:Extensions}
Let $L \subset A^*$ be a regular factor-closed language and $X \subset A^\Z$ the largest subshift with $\lang{X} \subset L$. Then there exists $C \geq 0$ such that $u w v \in L$ and $|u|, |v| \geq C$ implies $w \in \lang{X}$.
\end{lemma}

\begin{proof}
It is easy to see that $X = \{x \in A^\Z \;|\; \forall k: x|_{[-k, k]} \in L\}$. Take a nondeterministic finite-state automaton (NFA) for $L$ with $n$ states all of whose states are initial and final, and pick $C = n+1$. Suppose $uwv \in L$ and $|u|, |v| \geq C$, and pick any accepting path $P$ for $u w v$ in the automaton. By the pigeonhole principle, some state repeats in the length-$C$ prefix of $P$ corresponding to $u$, and the same is true for the suffix corresponding to $v$, so we obtain decompositions $u = u_1 u_2 u_3, v = v_1 v_2 v_3$ with $|u_2|, |v_2| > 0$ such that $u_2^k u_3 w v_1 v_2^k \in L$ for all $k \in \N$. Then ${}^\infty u_2 u_3 w v_1 v_2^\infty \in X$, so that $w \in \lang{X}$.
\end{proof}

By the proof, we can always pick $C = n+1$ where $n$ is the number of states in any NFA accepting $L$ with all states initial and final. Often we can do better.

\begin{lemma}
\label{lem:StableImplies}
Let $f : A^{\Z^2} \to A^{\Z^2}$ be a cellular automaton with neighborhood $[-r,r]^2$. 
Let $F$ be the patterns $P \in A^{[-r, r]^2}$ that map to a nonzero symbol in the local rule of $f$. If $F$ has one-sided stable traces, then there exists $c \in \N$ such that
whenever $P \in A^{[0,M) \times [0,N)}$ and $\pad^c(P)$ admits a preimage $Q$, then $\conf(P)$ admits a preimage $x \in A^{\Z^2}$ with
\[ x|_{[-r,M+r) \times [-r,N+r)} = Q|_{[-r,M+r) \times [-r,N+r)}. \]
\end{lemma}

We use the same padding $c$ on all sides for notational convenience only. For the Game of Life this does not change anything, but for less symmetric CA and CA where the constant $C$ in the proof is not equal to $0$, one may optimize this by using different paddings on all sides.


\begin{proof}
Denote $n = 2r$ and let $\ell \in \N$ be such that $S_{n,\ell}(F) = S_{n,\ell+1}(F)$, so that $S_{n,\ell}(F) = \one{T}_n(F)$ by Lemma~\ref{lem:StableImpliesActualTrace}.
Consider the set
\[
  L_{n,\ell}(F) = \{ Q|_{[0,k) \times [0,n)} \;|\; Q \in A^{[0,k) \times [0,n+\ell)}, \forall P \in F : P \not\sqsubset Q \}
\]
of height-$n$ rectangular patterns that can be extended upward by $\ell$ steps into a pattern not containing any pattern from $F$ as a subpattern.
We consider it as a language over the alphabet $A^n$.
Since it is regular and factor-closed, and the largest subshift whose language is contained in $L_{n,\ell}(F)$ is clearly $S_{n,\ell}(F)$, by the previous lemma there exists $C \in \N$ such that if a word of $L_{n,\ell}(F)$ is extendable by $C$ steps in both directions, it occurs in $S_{n,\ell}(F) = T_n(F)$.

Let $a = C + \ell$ and $b = \ell$, and suppose a pattern $P \in A^{[-a, M+a) \times [-b, N+b)}$ has its support contained in $D = [0,M) \times [0, N)$ and admits a preimage. Let $x$ be a configuration with $f(x)|_{[-a, M+a) \times [-b, N+b)} = P$. Then the $r$ top rows of $D$ and the $r$ rows above them satisfy $\sigma^{(0,N-r-1)}(x)|_{[-a-r, M+a+r) \times [0,2r)} \in L_{n,\ell}(F)$, therefore $\sigma^{(0, N-r-1)}(x)|_{[-r-\ell, M+r+\ell) \times [0,2r)} \in \lang{S_{n,\ell}(F)} = \lang{T_n(F)}$ by the assumption on $C$. Hence we can extend this pattern into a stripe of $T_n(F)$ and then extend the upper half-plane into one that maps to $0$s, obtaining a configuration $x^1 \in A^{\Z^2}$ with $x^1|_{[-a,M+a)} = x|_{[-b,N+r)}$ and $x^1|_{\Z \times (N-r, \infty)}$ not containing occurrences of any pattern in $F$.

Perform the same operation symmetrically on the south border of $D$ in $x^1$ to obtain a configuration $x^2$. 
Then the support of $f(x^2)$ is contained in $([0, M) \times [0, N)) \cup (R \times [0, N))$ where $R = (-\infty, -\ell] \cup [M+\ell, \infty)$.
We can now apply the exact same argument on the west and east borders of $D$ in $x^2$ (in either order) to obtain a configuration $x^3$ such that the support of $f(x^3)$ is contained in $[0, M) \times [0, N)$ and $f(x^3)|_{[0,M) \times [0,N)} = f(x)|_{[0,M) \times [0,N)}$. Picking $c = \max(a, b) = C + \ell$, this argument shows the first claim.

For the claim that $\FinG(f)$ many-one reduces to orphans in polynomial time, given an element $y$ of $\Fin(A)$ with support contained in $[0, M) \times [0, N)$, we simply note that the pattern $y|_{[-c, M+c) \times [-c, N+c)}$ is an orphan if and only if $y \in \FinG(f)$.
\end{proof}

Since orphans (of any finite shape) are in co-NP for any cellular automaton, the previous lemma gives another verification algorithm for showing that $\FinG(f)$ is co-NP. It is easier to implement as that given by Lemma~\ref{lem:PeriodizableImplies} when the conditions of both results apply. However, the certificate is less useful: unlike in Lemma~\ref{lem:PeriodizableImplies}, the certificate does not describe an actual preimage of the finite-support configuration. Together, Lemma~\ref{lem:PeriodizableImplies} and Lemma~\ref{lem:StableImplies} imply (when the assumptions hold) that whenever a pattern extends to a non-orphan with a large enough zero-padding, the corresponding finite-support configuration admits a semilinear preimage. This will be illustrated in the next section. 

\section{Application to the Game of Life}

\begin{theorem}
\label{thm:GoLPeriodizable}
The preimage SFT of $0^{\Z^2}$ in the Game of Life has stable and periodizable one-sided traces.
\end{theorem}

\begin{proof}
Let $g : \{0,1\}^{\Z^2} \to \{0,1\}^{\Z^2}$ be the Game of Life. Let $F$ be the natural forbidden patterns for $g^{-1}(0^{\Z^2})$. Since $g$ has radius $1$, we pick $n = 2$. Clearly $F = {\circlearrowright^i}(F)$ for all $i$, so it is enough to find $k, \ell, p$ such that $\one{S}_{n, \ell}(F) \subset \one{P}_{n,k,p}(F)$. One can show that the choice $\ell = 4, k = 3, p = 3$ works, by computer.
\end{proof}

Our program in Appendix~\ref{sec:Program} verifies this result.

By using the minimal automata for the languages of the approximate traces, one can check\footnote{Once the word $w$ is given, one can prove by hand straight form the definitions that it separates the regular languages, as a constraint-solving puzzle.} that
\[ w = \begin{smallmatrix}
1 & 1 & 0 & 0 & 1 & 0 & 0 & 0 & 0 & 0 & 0 & 0 & 0 & 0 & 1 & 0 & 0 & 0 & 1 & 0 & 1 & 0 & 1 & 0 & 1 & 0 & 0 & 0 & 1 & 0 \\
0 & 0 & 1 & 0 & 0 & 0 & 0 & 1 & 1 & 1 & 0 & 0 & 1 & 0 & 1 & 0 & 1 & 0 & 1 & 0 & 0 & 0 & 1 & 0 & 0 & 0 & 1 & 0 & 0 & 0
\end{smallmatrix} \in \lang{S_{2, 3}(F)} \setminus \lang{S_{2, 4}(F)}, \]
i.e. $w$ can be extended infinitely on both sides so that the resulting configuration can be continued by three rows upward without introducing a $1$ in the Game of Life image, but not by four rows. Since these subshifts are not equal, $\ell = 4$ is the minimal value we can use. One can similarly show that $k = 3, p = 3$ are also optimal; $p = 3$ is optimal by using the word from the proof of Proposition~\ref{prop:AsympPeriodicNotDense} below. The optimality of $k = 3$ can be shown similarly as that of $\ell$ by studying the minimal automata.

\begin{lemma}
\label{lem:GoLConstant}
In the situation of the Game of Life, we have $L_{2, 4}(F) = \lang{S_{2, 4}(F)}$, i.e.\ in Lemma~\ref{lem:Extensions} applied to $L_{2, 4}(F)$ one can pick $C = 0$.
\end{lemma}

\begin{proof}
  It can be checked, either with a case analysis or by computer, that $S_{2, 1}(F)$ contains all words of length $6$ over its alphabet.
  Given a word $w \in L_{2,4}(F)$, let $P$ be a height-$6$ rectangle whose two bottom rows form $w$ and that contains no pattern from $F$.
  The two rightmost columns of $P$ are in ${\circlearrowright}(S_{2, 1}(F))$, so $P$ can be extended to the right by one column without introducing a pattern of $F$.
  Symmetrically, it can be extended to the left, and thus $w$ is extendable indefinitely in both directions within $L_{2,4}(F)$.
\end{proof}

Our program in Appendix~\ref{sec:Program} also verifies Lemma~\ref{lem:GoLConstant}.


\begin{theorem}
\label{thm:Main}
Let $g$ be the Game of Life. Then
\begin{itemize}
\item semilinear configurations are dense in $g^{-1}(y)$ for every finite-support configuration $y$,
\item if $P \in \{0,1\}^{[0,M) \times [0,N)}$ and $\pad^4(P)$ admits preimage $Q$, then $\conf(P)$ is not a Garden of Eden, and admits a preimage $x \in \{0,1\}^{\Z^2}$ with
\[ x|_{[-1,M] \times [-1,N]} = Q|_{[-1,M] \times [-1,N]} \]
\item $\FinG(g)$ reduces in polynomial time to orphans, in particular $\FinG(g)$ is in co-NP.
\end{itemize}
\end{theorem}

\begin{proof}
The first claim follows from Theorem~\ref{thm:GoLPeriodizable} and Lemma~\ref{lem:PeriodizableImplies}. The second and third follow from Theorem~\ref{thm:GoLPeriodizable} (whose proof gives stability at $S_{2, 4}(F)$), and then applying the proof of Lemma~\ref{lem:StableImplies} using the constant $C = 0$ from Lemma~\ref{lem:GoLConstant}.
\end{proof}

The constants are not included in the statements of the lemmas for simplicity, but their values are explicitly stated in the proofs.

In the above theorem, we have a constant bound on the periods of the semilinear sets, by the proof of Lemma~\ref{lem:PeriodizableImplies}. The vertical period $q! = (p |A|^{2 n (k+p)})! = 50331648!$ stated in the proof of Lemma~\ref{lem:PeriodizableImplies} is not practically usable. By analyzing the situation a bit more carefully, we see that it is better to use separate periods of at most $p |A|^{n (k+p)} = 12288$ on the west and east borders of $A_{\mathrm{north}}$ and $A_{\mathrm{south}}$. This is already usable, and should be hugely improvable by further analysis.

\begin{example}
\label{ex:MakingSemilinear}
We illustrate how to construct a semilinear preimage of a finite-support configuration from a preimage of a rectangular pattern containing its support. The following shows a pattern $P$ and its thickness-$4$ padding $\pad^4(P)$: 
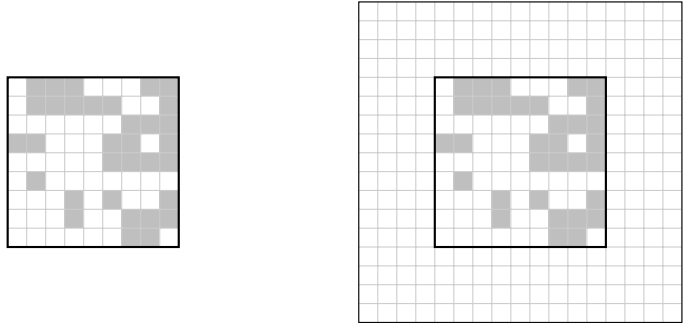
\begin{center}
\pgfplotstableread{image_pattern_P.cvs}{\matrixfile}
\begin{tikzpicture}[scale = 0.25]

  \pgfplotstablegetrowsof{\matrixfile} 
  \pgfmathtruncatemacro{\totrow}{\pgfplotsretval}
  \pgfplotstablegetcolsof{\matrixfile} 
  \pgfmathtruncatemacro{\totcol}{\pgfplotsretval}
  
  \pgfplotstableforeachcolumn\matrixfile\as\col{
    \pgfplotstableforeachcolumnelement{\col}\of\matrixfile\as\colcnt{%
      \ifnum\colcnt=0
        \fill[white]($ -\pgfplotstablerow*(0,\cellht) + \col*(\cellwd,0) $) rectangle +(\cellwd,\cellht);
      \fi
      \ifnum\colcnt=1
        \fill[gray!50!white]($ -\pgfplotstablerow*(0,\cellht) + \col*(\cellwd,0) $) rectangle+(\cellwd,\cellht);
      \fi
      \ifnum\colcnt=2
        \fill[white]($ -\pgfplotstablerow*(0,\cellht) + \col*(\cellwd,0) $) rectangle +(\cellwd,\cellht);
        \draw ($ -\pgfplotstablerow*(0,\cellht) + (\cellwd*0.5,\cellht*0.5) + \col*(\cellwd,0) $) circle (0.1);
      \fi
      \ifnum\colcnt=3
        \fill[gray!50!white]($ -\pgfplotstablerow*(0,\cellht) + \col*(\cellwd,0) $) rectangle+(\cellwd,\cellht);
        \draw ($ -\pgfplotstablerow*(0,\cellht) + (\cellwd*0.5,\cellht*0.5) + \col*(\cellwd,0) $) circle (0.1);
      \fi
    }
  }
  \draw[black!20!white] (0, -8) grid (9,1);
  \draw[thick,black] (0, -8) rectangle (9,1);
    \draw[white] (-4, -12) rectangle (13,5);
\end{tikzpicture}
\hspace{1cm}
\pgfplotstableread{image_pattern_P.cvs}{\matrixfile}
\begin{tikzpicture}[scale = 0.25]

  \pgfplotstablegetrowsof{\matrixfile} 
  \pgfmathtruncatemacro{\totrow}{\pgfplotsretval}
  \pgfplotstablegetcolsof{\matrixfile} 
  \pgfmathtruncatemacro{\totcol}{\pgfplotsretval}
  
  \pgfplotstableforeachcolumn\matrixfile\as\col{
    \pgfplotstableforeachcolumnelement{\col}\of\matrixfile\as\colcnt{%
      \ifnum\colcnt=0
        \fill[white]($ -\pgfplotstablerow*(0,\cellht) + \col*(\cellwd,0) $) rectangle +(\cellwd,\cellht);
      \fi
      \ifnum\colcnt=1
        \fill[gray!50!white]($ -\pgfplotstablerow*(0,\cellht) + \col*(\cellwd,0) $) rectangle+(\cellwd,\cellht);
      \fi
      \ifnum\colcnt=2
        \fill[white]($ -\pgfplotstablerow*(0,\cellht) + \col*(\cellwd,0) $) rectangle +(\cellwd,\cellht);
        \draw ($ -\pgfplotstablerow*(0,\cellht) + (\cellwd*0.5,\cellht*0.5) + \col*(\cellwd,0) $) circle (0.1);
      \fi
      \ifnum\colcnt=3
        \fill[gray!50!white]($ -\pgfplotstablerow*(0,\cellht) + \col*(\cellwd,0) $) rectangle+(\cellwd,\cellht);
        \draw ($ -\pgfplotstablerow*(0,\cellht) + (\cellwd*0.5,\cellht*0.5) + \col*(\cellwd,0) $) circle (0.1);
      \fi
    }
  }
  \draw[black!20!white] (-4, -12) grid (13,5);
  \draw[thick,black] (0, -8) rectangle (9,1);
  \draw[black] (-4, -12) rectangle (13,5);
\end{tikzpicture}
\end{center}
We want to know if the finite-support configuration corresponding to $P$ is in the image of the Game of Life. The first step is to find a preimage for the padded pattern, for example using a SAT solver.\footnote{PicoSAT \cite{Bi08} solves this particular instance in under a second. The solution shown here was not constructed by a SAT solver.} One possible preimage (for the thickness four padding) is the following one.
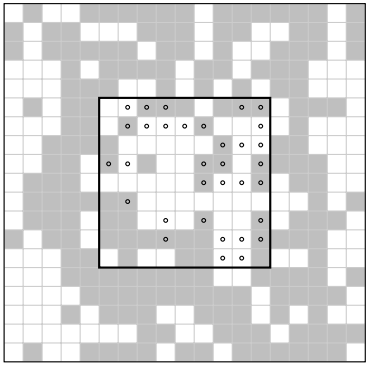
\begin{center}
\pgfplotstableread{example_pattern_bad.cvs}{\matrixfile}
\begin{tikzpicture}[scale = 0.25]

  \pgfplotstablegetrowsof{\matrixfile} 
  \pgfmathtruncatemacro{\totrow}{\pgfplotsretval}
  \pgfplotstablegetcolsof{\matrixfile} 
  \pgfmathtruncatemacro{\totcol}{\pgfplotsretval}
  
  \pgfplotstableforeachcolumn\matrixfile\as\col{
    \pgfplotstableforeachcolumnelement{\col}\of\matrixfile\as\colcnt{%
      \ifnum\colcnt=0
        \fill[white]($ -\pgfplotstablerow*(0,\cellht) + \col*(\cellwd,0) $) rectangle +(\cellwd,\cellht);
      \fi
      \ifnum\colcnt=1
        \fill[gray!50!white]($ -\pgfplotstablerow*(0,\cellht) + \col*(\cellwd,0) $) rectangle+(\cellwd,\cellht);
      \fi
      \ifnum\colcnt=2
        \fill[white]($ -\pgfplotstablerow*(0,\cellht) + \col*(\cellwd,0) $) rectangle +(\cellwd,\cellht);
        \draw ($ -\pgfplotstablerow*(0,\cellht) + (\cellwd*0.5,\cellht*0.5) + \col*(\cellwd,0) $) circle (0.1);
      \fi
      \ifnum\colcnt=3
        \fill[gray!50!white]($ -\pgfplotstablerow*(0,\cellht) + \col*(\cellwd,0) $) rectangle+(\cellwd,\cellht);
        \draw ($ -\pgfplotstablerow*(0,\cellht) + (\cellwd*0.5,\cellht*0.5) + \col*(\cellwd,0) $) circle (0.1);
      \fi
    }
  }
  \draw[black!20!white] (0, -18) grid (19,1);
  \draw[black] (0, -18) rectangle (19,1);
  \draw[thick,black] (5, -13) rectangle (14,-4);
\end{tikzpicture}
\end{center}
The occurrences of $1$s in the image correspond to the black dots, and the $1$s in the preimage are gray tiles. One can check that this preimage cannot be extended to the south without introducing a $1$ in the image, can be extended to the west by one line (but not two), and to the east by two lines (but not three).

Lemma~\ref{lem:StableImplies} implies that, since we found a preimage for the thickness-$4$ padding, there exists a preimage for the finite-support configuration $y = \conf(P)$. Figure~\ref{fig:MakingSemilinear} shows such a preimage, and the process of extracting a semilinear preimage from it: The figures denote central $[-16, 16]^2$ patterns of configurations $x, x^2, x^3$, and their common Game of Life image $y$, whose support contained in $[-N, N]^2 = [-4, 4]^2$ is $P$. These correspond to $x, x^2, x^3, y$ in the proof of Lemma~\ref{lem:PeriodizableImplies}. The $1$-cells of the patterns $x, x^2, x^3$ are gray tiles, the $1$-cells in $y$ are black dots, and the inner square is $[-N, N]^2$. The first configuration $x$ was sampled together with $y$, by picking $1$-cells of $x$ from a Bernoulli distribution, and then adding $1$-cells outside $[-N, N]^2$ until the image $y = g(x)$ outside $[-N, N]^2$ contained only $0$-cells. We have $S_{2,4}(F) \subset P_{2,3,3}(F)$, and $r = 1, n = 2, \ell = 4, k = 3, p = 3$ in the notation of the proof of Lemma~\ref{lem:PeriodizableImplies}.

The second configuration, $x^2$ was obtained by periodizing the contents below and above $[-N, N]^2$ using $S_{2,4}(F) \subset P_{2,3,3}(F)$, and the periodic half-planes $\Z \times [9, \infty)$ and $\Z \times (-\infty, -9]$ are delineated (here $N+k+r+1 = 9$). The continuations picked are the lexicographically minimal ones (in the visible area) with preperiod $3$ and period $3$, moving counterclockwise. 
The third configuration $x^3$ is obtained from $x^2$ by periodizing on the west and east, again using $S_{2,4}(F) \subset P_{2,3,3}(F)$ and picking lexicographically minimal continuations. The regions
\[ A_{\mathrm{north}} = [-N-r, N+r] \times [N+r+k+1, \infty) = [-5, 5] \times [9, \infty), \]
\[ A_{\mathrm{south}} =[-N-r, N+r] \times (-\infty, -N-r-k-1] = [-5, 5] \times (-\infty, -9], \]
\[ A_{\mathrm{west}} = (-\infty, -N-r-k-1] \times \Z = (-\infty, -9] \times \Z, \]
\[ A_{\mathrm{east}} = [N+r+k+1, \infty) \times \Z = [9, \infty) \times \Z, \]
are delineated.

\begin{figure}[h!]
\begin{center}
\begin{subfigure}[b]{0.33\textwidth}
 \pgfplotstableread{example_pattern.cvs}{\matrixfile}
\begin{tikzpicture}[scale = 0.11]

  \pgfplotstablegetrowsof{\matrixfile} 
  \pgfmathtruncatemacro{\totrow}{\pgfplotsretval}
  \pgfplotstablegetcolsof{\matrixfile} 
  \pgfmathtruncatemacro{\totcol}{\pgfplotsretval}
  
  \pgfplotstableforeachcolumn\matrixfile\as\col{
    \pgfplotstableforeachcolumnelement{\col}\of\matrixfile\as\colcnt{%
      \ifnum\colcnt=0
        \fill[white]($ -\pgfplotstablerow*(0,\cellht) + \col*(\cellwd,0) $) rectangle +(\cellwd,\cellht);
      \fi
      \ifnum\colcnt=1
        \fill[gray!50!white]($ -\pgfplotstablerow*(0,\cellht) + \col*(\cellwd,0) $) rectangle+(\cellwd,\cellht);
      \fi
      \ifnum\colcnt=2
        \fill[white]($ -\pgfplotstablerow*(0,\cellht) + \col*(\cellwd,0) $) rectangle +(\cellwd,\cellht);
        \draw ($ -\pgfplotstablerow*(0,\cellht) + (\cellwd*0.5,\cellht*0.5) + \col*(\cellwd,0) $) circle (0.1);
      \fi
      \ifnum\colcnt=3
        \fill[gray!50!white]($ -\pgfplotstablerow*(0,\cellht) + \col*(\cellwd,0) $) rectangle+(\cellwd,\cellht);
        \draw ($ -\pgfplotstablerow*(0,\cellht) + (\cellwd*0.5,\cellht*0.5) + \col*(\cellwd,0) $) circle (0.1);
      \fi
    }
  }
  \draw[black!20!white] (0, -32) grid (33,1);
  \draw[black] (0, -32) rectangle (33,1);
  \draw[thick,black] (12, -20) rectangle (21,-11);
\end{tikzpicture}
\end{subfigure}
\begin{subfigure}[b]{0.33\textwidth}
\pgfplotstableread{example_pattern_per_sn.cvs}{\matrixfile}
\begin{tikzpicture}[scale = 0.11]

  \pgfplotstablegetrowsof{\matrixfile} 
  \pgfmathtruncatemacro{\totrow}{\pgfplotsretval}
  \pgfplotstablegetcolsof{\matrixfile} 
  \pgfmathtruncatemacro{\totcol}{\pgfplotsretval}
  
  \pgfplotstableforeachcolumn\matrixfile\as\col{
    \pgfplotstableforeachcolumnelement{\col}\of\matrixfile\as\colcnt{%
      \ifnum\colcnt=0
        \fill[white]($ -\pgfplotstablerow*(0,\cellht) + \col*(\cellwd,0) $) rectangle +(\cellwd,\cellht);
      \fi
      \ifnum\colcnt=1
        \fill[gray!50!white]($ -\pgfplotstablerow*(0,\cellht) + \col*(\cellwd,0) $) rectangle+(\cellwd,\cellht);
      \fi
      \ifnum\colcnt=2
        \fill[white]($ -\pgfplotstablerow*(0,\cellht) + \col*(\cellwd,0) $) rectangle +(\cellwd,\cellht);
        \draw ($ -\pgfplotstablerow*(0,\cellht) + (\cellwd*0.5,\cellht*0.5) + \col*(\cellwd,0) $) circle (0.1);
      \fi
      \ifnum\colcnt=3
        \fill[gray!50!white]($ -\pgfplotstablerow*(0,\cellht) + \col*(\cellwd,0) $) rectangle+(\cellwd,\cellht);
        \draw ($ -\pgfplotstablerow*(0,\cellht) + (\cellwd*0.5,\cellht*0.5) + \col*(\cellwd,0) $) circle (0.1);
      \fi
    }
  }

  \draw[black!20!white] (0, -32) grid (33,1);
    
  \draw[black] (0, -32) rectangle (33,1);
  \draw[black,thick] (12, -20) rectangle (21,-11);
  
  \draw[black] (0, -7) -- (33, -7);
   \draw[black] (0, -24) -- (33, -24);
  
\end{tikzpicture}
\end{subfigure}
\begin{subfigure}[b]{0.32\textwidth}
\pgfplotstableread{example_pattern_per_snew.cvs}{\matrixfile}
\begin{tikzpicture}[scale = 0.11]

  \pgfplotstablegetrowsof{\matrixfile} 
  \pgfmathtruncatemacro{\totrow}{\pgfplotsretval}
  \pgfplotstablegetcolsof{\matrixfile} 
  \pgfmathtruncatemacro{\totcol}{\pgfplotsretval}
  
  \pgfplotstableforeachcolumn\matrixfile\as\col{
    \pgfplotstableforeachcolumnelement{\col}\of\matrixfile\as\colcnt{%
      \ifnum\colcnt=0
        \fill[white]($ -\pgfplotstablerow*(0,\cellht) + \col*(\cellwd,0) $) rectangle +(\cellwd,\cellht);
      \fi
      \ifnum\colcnt=1
        \fill[gray!50!white]($ -\pgfplotstablerow*(0,\cellht) + \col*(\cellwd,0) $) rectangle+(\cellwd,\cellht);
      \fi
      \ifnum\colcnt=2
        \fill[white]($ -\pgfplotstablerow*(0,\cellht) + \col*(\cellwd,0) $) rectangle +(\cellwd,\cellht);
        \draw ($ -\pgfplotstablerow*(0,\cellht) + (\cellwd*0.5,\cellht*0.5) + \col*(\cellwd,0) $) circle (0.1);
      \fi
      \ifnum\colcnt=3
        \fill[gray!50!white]($ -\pgfplotstablerow*(0,\cellht) + \col*(\cellwd,0) $) rectangle+(\cellwd,\cellht);
        \draw ($ -\pgfplotstablerow*(0,\cellht) + (\cellwd*0.5,\cellht*0.5) + \col*(\cellwd,0) $) circle (0.1);
      \fi
    }
  }
  
  \draw[black!20!white] (0, -32) grid (33,1);
  \draw[black] (0, -32) rectangle (33,1);
  \draw[black,thick] (12, -20) rectangle (21,-11);
  
  \draw[black] (8, -32) -- (8, 1);
  \draw[black] (25, -32) -- (25, 1);
  
  \draw[black] (11, 1) -- (11, -7) -- (22, -7) -- (22, 1);
  \draw[black] (11, -32) -- (11, -24) -- (22, -24) -- (22, -32);
  
\end{tikzpicture}
\end{subfigure}
\end{center}
\caption{Turning an arbitrary preimage of a finite-support configuration into a semilinear one.}
\label{fig:MakingSemilinear}
\end{figure}
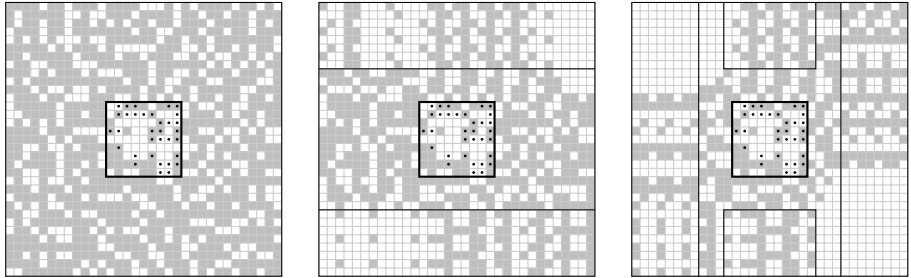

The configuration $x^4$ would be obtained by finding a repetition in the rows to the north and south of the pattern. Only the stripe at the east border of $A_{\mathrm{north}}$ is not yet periodic, and we can simply replace the entire quadrant with $1$-cells to obtain a semilinear preimage. \qee
\end{example}

\section{Questions and additional observations}

The following proposition shows that (the first claim in) Theorem~\ref{thm:Main} is no longer true if semilinear configurations are replaced by finite-support or co-finite-support configurations. Say a configuration is \emph{asymptotically horizontally $N$-periodic} if it is asymptotic to a configuration with period $(N,0)$.

\begin{proposition}
\label{prop:AsympPeriodicNotDense}
Let $g$ be the Game of Life and $N \in \N$. Then asymptotically horizontally $N$-periodic configurations are not dense in $g^{-1}(0^{\Z^2})$.
\end{proposition}

Obviously the same is true for vertically periodic configurations.

\begin{proof}
One can verify that $g^{-1}(0^{\Z^2})$ contains the pattern
\[ P_n = 
 \dd{0}{0} \; \dd{1}{1} \; \dd{0}{0} \; \dd{0}{1} \; \dd{0}{0} \; \dd{1}{1} \;
\dd{0}{0} \; \dd{1}{0} \; \dd{0}{0} \; \dd{1}{1} \; \dd{0}{0} \; \dd{1}{0} \;
\dd{0}{0} \; \dd{0}{1} \; \dd{0}{0} \; \dd{1}{0} \; \dd{0}{0} \; \dd{0}{1} 
\left( \dd{0}{0} \; \dd{0}{1} \right)^n \dd{0}{0} \]
for any $n \in \N$, by continuing it left and right with $0$s, and then extending each column with period-$3$ so that the repeating pattern has an even number of $1$s.

It turns out that this is the only possible continuation of the columns intersecting the pattern: At an occurrence of
\[ \begin{matrix}
0 & 1 & 0 & 0 & 0 & 1 & 0 \\
0 & 1 & 0 & 1 & 0 & 1 & 0 
\end{matrix} \]
one can easily deduce that the only possible continuation downward is $0001000$, and this forces the contents of the next row. The pattern is chosen so that this subpattern is seen on the first $3$ iterations, after which we repeat the original pattern, which determines the entire column.

To prove the claim, pick $n > N$. Then no preimage of $0^{\Z^2}$ containing $P_n$ has a horizontal $N$-period.
\end{proof}

The pattern above was found by generalizing the shortest word separating $S_{2,3}(F)$ from $S_{2,4}(F)$. It forces the period in both directions, thus cannot appear in a preimage of $0^{\Z^2}$ together with its $90$-degree rotation. Thus, it cannot directly be used to show that $g^{-1}(0^{\Z^2})$ does not have dense asymptotically doubly periodic points (when the period is not fixed). Indeed, we do not know whether asymptotically doubly periodic points are dense in $g^{-1}(0^{\Z^2})$.

Our first attempt at proving the decidability of the set of finite-support Gardens of Eden was to show that all finite-support configurations that have a preimage have a finite-support or co-finite-support preimage. This stays open.

\begin{question}
\label{q:DoesEvery}
Does every finite-support configuration in the image of the Game of Life have a finite-support preimage? A co-finite-support preimage? An asymptotically doubly periodic preimage with some fixed periods $(N,0)$, $(0,N)$?
\end{question}

The first and second subquestions are equivalent, as observed by user dvgrn in \cite{OrphanvsGoE}: in the SFT $g^{-1}(0^{\Z^2})$, the rectangular all-$1$ pattern with support $[-n+2, n+2]^2 \setminus [-n, n]^2$ can be glued to the all-$0$ pattern with support $[-n-c-2, n+c+2]^2 \setminus [-n-c, n+c]^2$, for any large enough $c$, and vice versa when $0$ and $1$ are exchanged.

Proposition~\ref{prop:AsympPeriodicNotDense} shows that ``yes'' answers in Question~\ref{q:DoesEvery} cannot be proved by only modifying an arbitrary preimage outside a finite region. These subquestions would be solved in the negative by finding a finite-support configuration $y \in \{0,1\}^{\Z^2}$ that is not a Garden of Eden but whose every preimage contains an occurrence of some $P_n$ from the proof of Proposition~\ref{prop:AsympPeriodicNotDense} outside the convex hull of the support of $y$.

Anecdotal evidence for Question~1 is that, as mentioned, the preimages shown for $P$ and its paddings in Example~\ref{ex:MakingSemilinear} were not found by a SAT solver, and in fact if the problem is fed into PicoSAT with a lexicographic order on the positions, the preimage that is found for the $4$-padding of $P$ in Example~\ref{ex:MakingSemilinear} always has two layers of $1$-cells in the two outermost layers, in other words finding a cofinite-support preimage for this particular pattern seems to be easier than finding a ``bad'' preimage (so we constructed one by other methods). If the problem is obfuscated (by shuffling the cell positions) before feeding it to the solver, it occasionally gives solutions that do not have $1$-cells on the border of the rectangle, but this still happens very rarely. 

Much of the study of the Game of Life concentrates on finite configurations. Thus, the following question seems very relevant, if the first subquestion of Question~\ref{q:DoesEvery} has a negative answer.

\begin{question}
Is it decidable whether a given finite-support configuration has a finite-support preimage?
\end{question}

Finite-support Gardens of Eden are $\Sigma^0_1$ (in the arithmetical hierarchy) and non-trivially decidable (by this paper), since if there is no preimage, there is an orphan. Finite-support configurations without a finite-support preimage, on the other hand, are $\Pi^0_1$, since we can prove the \emph{no}-instances by exhibiting a preimage, but there is no obvious reason there should be a finite certificate for not having one.

In Theorem~\ref{thm:Main}, the padding thickness $4$ is optimal if we are not allowed to change the preimage, by positioning any word of $\lang{S_{2, 3}(F)} \setminus \lang{S_{2, 4}(F)}$ on the boundary of the rectangle. We do not know if it is optimal if the preimage can be changed, i.e.\ we do not know whether there exist patterns $P$ such that $\pad^3(P)$ has a preimage pattern but $\conf(P)$ is a Garden of Eden. The optimal constant is at least $1$: the pattern
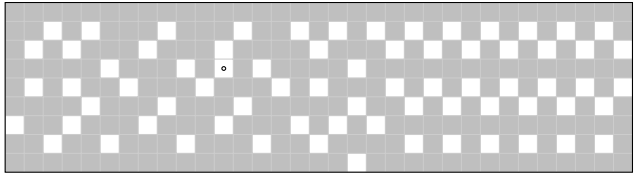
\begin{center}
\pgfplotstableread{orphan_after_pad.cvs}{\matrixfile}
\begin{tikzpicture}[scale = 0.25]

  \pgfplotstablegetrowsof{\matrixfile} 
  \pgfmathtruncatemacro{\totrow}{\pgfplotsretval}
  \pgfplotstablegetcolsof{\matrixfile} 
  \pgfmathtruncatemacro{\totcol}{\pgfplotsretval}
  
  \pgfplotstableforeachcolumn\matrixfile\as\col{
    \pgfplotstableforeachcolumnelement{\col}\of\matrixfile\as\colcnt{%
      \ifnum\colcnt=0
        \fill[white]($ -\pgfplotstablerow*(0,\cellht) + \col*(\cellwd,0) $) rectangle +(\cellwd,\cellht);
      \fi
      \ifnum\colcnt=1
        \fill[gray!50!white]($ -\pgfplotstablerow*(0,\cellht) + \col*(\cellwd,0) $) rectangle+(\cellwd,\cellht);
      \fi
      \ifnum\colcnt=2
        \fill[white]($ -\pgfplotstablerow*(0,\cellht) + \col*(\cellwd,0) $) rectangle +(\cellwd,\cellht);
        \draw ($ -\pgfplotstablerow*(0,\cellht) + (\cellwd*0.5,\cellht*0.5) + \col*(\cellwd,0) $) circle (0.1);
      \fi
      \ifnum\colcnt=3
        \fill[gray!50!white]($ -\pgfplotstablerow*(0,\cellht) + \col*(\cellwd,0) $) rectangle+(\cellwd,\cellht);
        \draw ($ -\pgfplotstablerow*(0,\cellht) + (\cellwd*0.5,\cellht*0.5) + \col*(\cellwd,0) $) circle (0.1);
      \fi
    }
  }
  \draw[black!20!white] (0, -8) grid (33,1);
  \draw[black] (0, -8) rectangle (33,1);
\end{tikzpicture}
\end{center}
was obtained by modifying one bit in the orphan of Banks \cite{Ga83} (marked by a black dot). It is not an orphan, but PicoSAT reports its thickness-$1$ zero-padding to be an orphan.

One may ask if it is important in Theorem~\ref{thm:Main} that $P$ is rectangular. For the result that finite-support Gardens of Eden are co-NP, this is not essential: as long as the inputs of size $n$ specify values only in a polynomial-sized rectangle in $n$, to prove that the configuration $\conf(P)$ corresponding to a given pattern $P$ is not a Garden of Eden, we can extend the $P$ to a rectangular one by adding $0$-cells and apply our methods. However, the second claim of the theorem fails for all convex shapes except rectangles, at large enough scales.

For a general $D \subset \Z^2$, we define the \emph{$0$-padding of thickness $C$} of $P \in A^D$ as the pattern $Q$ with domain $E = (D + [-C, C]^2) \cap \Z^2$, where $D + [-C, C]^2$ is interpreted in $\R^2$ and $[-C, C] \subset \R$ is a continuous interval, defined by $Q|_D = P$, $Q|_{E \setminus D} = 0^{E \setminus D}$.

\begin{proposition}
Let $g$ be the Game of Life and $N \in \N$. Let $K \subset \R^2$ be a compact convex set that is the closure of its interior. If $K$ is not a rectangle aligned with the standard axes of $\R^2$, then for all $C \in \R$, for all large enough $r > 0$, the $0$-padding of thickness $C$ of the all-zero pattern $P$ with shape $r K \cap \Z^2$ admits a preimage pattern $Q$, such that the subpattern of $Q$ that maps to $P$ does not extend to a preimage of $\conf(P)$.
\end{proposition}

The constant $C$ can be replaced by any sublinear function of $r$.

\begin{proof}
Among compact convex subsets of $\R^2$, the axis-aligned rectangles are exactly the sets where $(a, b), (c, d) \in K$ implies $(a, d), (c, b) \in K$. 
Thus, suppose $K$ is not a rectangle, and pick $(a, b) \neq (c, d)$ such that w.l.o.g.\ $(a, d) \notin K$. By the assumption that $K$ is the closure of its interior, $(a, d)$ has positive distance to $K$, and we may assume $(a, b), (c, d)$ are interior points by moving them slightly if necessary.
For any $M \in \N$, there then exists $r > 0$ such that the set $r K \cap \Z^2$ contains a translate of the square grid $[-M, M]^2$ whose convex hull in turn contains $(r a, r b)$.
The same is true for $(c, d)$.


Let $P$ be the all-$0$ pattern of shape $r K \cap \Z^2$, and let $Q$ be a preimage pattern of $\pad^C(P)$ that contains a copy of the pattern $P_0$ from the proof of Proposition~\ref{prop:AsympPeriodicNotDense} near $(r a, r b)$ and a copy of the $90$-degree rotation ${\circlearrowright}(P_0)$ near $(r c, r d)$.
As long as $r$ is large enough compared to $C$, such a preimage exists, since we have enough room to put these patterns near $(r a, r b)$ and $(r c, r d)$, and then we can continue them $3$-periodically in both directions and fill the rest of $Q$ with $0$-cells.
The restriction of $Q$ to $r K \cap \Z^2$ does not extend to a preimage of $0^{\Z^2}$, since the $3$-periodic patterns forced by the copies of $P_0$ would intersect near $(r a, r d)$.
\end{proof}





It is also natural to ask whether the complexity-theoretic aspects of our results are optimal.
We have shown that for the Game of Life, the finite-support Gardens of Eden are polynomial-time reducible to orphans. For the other direction, we do not know any polynomial-time reduction, thus orphans could in principle be harder than finite-support Gardens of Eden. 
We conjecture that both problems are hard for co-NP, in particular are equally hard by our results.

\begin{conjecture}
Finite-support Gardens of Eden and rectangular orphans for the Game of Life are co-NP-complete (under the encodings of this paper, for polynomial time many-one reductions).
\end{conjecture}

The doubly periodic Gardens of Eden give another natural decision problem, and we conjecture that it is not computable.
This would in particular imply that not all semilinear configurations admit semilinear preimages. 

\begin{conjecture}
The doubly periodic Gardens of Eden for the Game of Life are an undecidable set. 
\end{conjecture}

One possible method of proving these conjectures would be to encode arbitrary Wang tiles into preimages of rectangular patterns, and reduce to the co-NP-complete problem of untileability of a rectangular region~\cite{GaJo79}, or to the undecidable problem of tileability of the infinite plane~\cite{Be66}.

In the encoding we use, the input specifies the values in an area of polynomial size. One may wonder what happens if this input is given in other ways.

\begin{question}
Given a finite tuple of vectors $(v_1, ..., v_n)$ written in binary, what is the computational complexity of checking whether the finite-support configuration with support $\{v_i \;|\; i = 1, ..., n\}$ is a Garden of Eden?
\end{question}

One may wonder if the Game of Life is special, or whether these results are true for a larger class of CA. The Game of Life has strong symmetry properties, so it is a natural candidate to look at, and possibly the automata-theoretic problems are particularly easy for this reason. However, as far as we know there is no reason to believe one-step properties of the Game of Life should be special among, for example, totalistic radius-$1$ rules with the Moore neighborhood. It would be interesting to go through a larger set of rules, and analyze the stability and periodizability properties of their traces. Of course, if the trace of a given SFT happens to be non-sofic (and in particular not stable), then our methods cannot be applied to it. Even if the trace is sofic but not stable, we do not have a general method of determining it.

We also do not know if our results apply to powers of the Game of Life.
If finite-support configurations that are not Gardens of Eden for the Game of Life would always have finite-support preimages, or if semilinear such configurations would always have semilinear preimages, then we would obtain decidability of finite-support Gardens of Eden for all powers of the Game of Life. 
It is also unknown if all powers of the Game of Life have different sets of Gardens of Eden, i.e.\ whether the Game of Life is \emph{stable} in the terminology of \cite{Ma95}; according to the LifeWiki website the first $4$ powers have been separated in this sense~\cite{LW-GFP19}.

\bibliographystyle{plain}
\bibliography{../../../bib/bib}{}

\newpage
\appendix

\section{Stable periodizable traces and $C = 0$ by computer}
\label{sec:Program}

\begin{Verbatim}
import time, collections


def valid_block(top_row, mid_row, bot_row, pos):
    """
        Is the 3x3 pattern at a given position of a 3-row rectangle valid?
        I.e. does it map to 0 in GoL?
    """
    mid_cell = mid_row[pos+1]
    ring_sum = sum(top_row[pos:pos+3]) + mid_row[pos] + mid_row[pos+2] + sum(bot_row[pos:pos+3])
    if mid_cell:
        return ring_sum < 2 or ring_sum > 3
    else:
        return ring_sum != 3

def iter_allowed_prefix(top_row, bot_row, prefix, end_period=False):
    """
        Iterate over locally legal words that fit to the bottom of the given rectangle and begin
	with the given prefix of length 2. If end_period is true, require further that the entire
	pattern can be continued to the right with period 3.
    """
    assert len(prefix) == 2
    length = len(top_row)
    assert 2 <= length == len(bot_row)
    if length == 2:
        yield prefix
    else:
        # Recursively pick a word of length one less and extend in both ways
        for word in iter_allowed_prefix(top_row[:-1], bot_row[:-1], prefix, end_period=False):
            for sym in range(2):
                new_row = word + (sym,)
                if end_period:
                    top_per = top_row + top_row[-3:-1]
                    bot_per = bot_row + bot_row[-3:-1]
                    new_per = new_row + new_row[-3:-1]
                    if all(valid_block(top_per, bot_per, new_per, length-i-1) for i in range(3)):
                        yield new_row
                elif valid_block(top_row, bot_row, new_row, length-3):
                    yield new_row

def encode_bin(word):
    """
        Encode a binary vector into an integer.
    """
    return sum(a * 2**i for (i,a) in enumerate(word))

def decode_bin(num, n):
    """
        Decode an integer into a pair of length-n binary tuples.
    """
    ret = []
    for _ in range(2*n):
        ret.append(num%2)
        num = num // 2
    return tuple(ret[:n]), tuple(ret[n:])

def right_trace_aut(radius, end_period=False):
    """
        NFA for one-sided vertical trace of width 2 and given right radius, possibly with periodic
        end. Alphabet is binary tuples of length 2, transitions go southward. States are
        rectangular patterns of height 2 and width radius+2, compressed into bytestrings.
        Return also C, the amount of extension needed to guarantee a word to be infinitely extendable.
    """
    alph = [(0,0),(0,1),(1,0),(1,1)]
    width = radius + 2
    trans_dict = {}
    states = set(range(2**(width*2)))
    byte_width = (width*2 + 7) // 8

    # Construct transitions
    for label in states:
        for sym in alph:
            res = []
            top_row, bot_row = decode_bin(label, width)
            for word in iter_allowed_prefix(top_row, bot_row, sym, end_period=end_period):
                res.append(encode_bin(bot_row + word))
            trans_dict[(label, sym)] = res

    # Remove states that lead to dead ends or can't be reached after some number of steps, compute C
    pad_needed = 0
    while True:
        new_states = set(st2
                         for st in states
                         for sym in alph
                         for st2 in trans_dict[(st, sym)])
        new_states = set(st
                         for st in new_states
                         if any(st2 in states
                                for sym in alph
                                for st2 in trans_dict[(st, sym)]))
        if new_states == states:
            break
        states = new_states
        pad_needed += 1

    # Update transition dict based on removed states, encode states as byte vectors
    trans_dict = {(st.to_bytes(byte_width, byteorder='big'),
                   sym)
                  :
                  [st2.to_bytes(byte_width, byteorder='big')
                   for st2 in trans_dict[(st, sym)]
                   if st2 in states]
                  for (st, sym) in trans_dict
                  if st in states}
    
    return (trans_dict, [st.to_bytes(byte_width, byteorder='big') for st in states], pad_needed)

def determinize(trans_dict, states, alph):
    """
        Determinize a given NFA using the powerset construction.
        It is assumed that all its states are initial and final.
    """

    # Maintain sets of seen and unprocessed state sets, and integer labels for seen sets
    init_st = frozenset(states)
    seen = {init_st : 0}
    finals = set([0])
    frontier = collections.deque([(init_st, 0)])

    det_trans = {}
    num_seen = 1

    while frontier:
        # Pick an unprocessed state set, go over its successors
        st_set, st_num = frontier.pop()
        for sym in alph:
            new_st_set = frozenset(st2 for st in st_set for st2 in trans_dict[(st, sym)])
            if new_st_set in seen:
                new_num = seen[new_st_set]
            else:
                # Pick a new label for the set
                new_num = num_seen
                num_seen += 1
                frontier.append((new_st_set, new_num))
                seen[new_st_set] = new_num
                # All nonempty sets of states are final
                if new_st_set:
                    finals.add(new_num)
            # Transitions are stored using the integer labels
            det_trans[(st_num, sym)] = new_num

    return det_trans, set(range(num_seen)), 0, finals
    
def minimize(trans_dict, states, init_st, final_states, alph):
    """
        Minimize a DFA using Moore's algorithm.
        It is assumed that all states are reachable.
    """

    # Maintain a coloring of the states; states with different colors are provably non-equivalent
    coloring = {}
    colors = set()
    for st in states:
        if st in final_states:
            coloring[st] = 1
            colors.add(1)
        else:
            coloring[st] = 0
            colors.add(0)
    num_colors = len(colors)

    # Iteratively update coloring based on the colors of successors
    while True:
        # First, use tuples of colors as new colors
        new_coloring = {}
        new_colors = set()
        for st in states:
            new_color = (coloring[st],) + tuple(coloring[trans_dict[(st, sym)]] for sym in alph)
            new_coloring[st] = new_color
            new_colors.add(new_color)
        # Then, encode new colors as integers
        color_nums = { color : i for (i, color) in enumerate(new_colors) }
        new_coloring = { st : color_nums[color] for (st, color) in new_coloring.items() }
        new_num_colors = len(new_colors)
        # If strictly more colors were needed, repeat
        if num_colors == new_num_colors:
            break
        else:
            colors = new_colors
            coloring = new_coloring
            num_colors = new_num_colors

    # Compute new transition function and state set
    new_trans_dict = {}
    for st in states:
        for sym in alph:
            new_trans_dict[(new_coloring[st], sym)] = new_coloring[trans_dict[(st, sym)]]
    
    new_final_states = set(new_coloring[st] for st in final_states)
    return new_trans_dict, set(new_coloring.values()), new_coloring[init_st], new_final_states
            

def diff_nonempty(dict_A, init_A, sink_A, dict_B, states_B, alph, track=False, verbose=False):
    """
        Is A\B nonempty for DFA A and NFA B?
        If track is True, return shortest word in A\B or None.
        If track is False, return True/False.
        Assumes A has a unique nonfinal state, which is a sink.
        Assumes all states of B are initial and final, and labeled by bytestrings.
        To preserve memory, states are stored in a compressed form.
    """

    # Compress states for a slight decrease in memory use
    def compre(x,y):
        return (x, b"".join(sorted(y)))

    clock = time.perf_counter()

    # Keep track of (compressed) processed states of pair automaton, B is implicitly determinized
    inits_B = frozenset(states_B)
    frontier = set([(init_A, inits_B, compre(init_A, inits_B))])
    if not track:
        reachables = set(c for (_,_,c) in frontier)
    else:
        reachables = {c : [] for (_,_,c) in frontier}

    # Process all reachable pairs in depth-first order, stopping if A accepts but B does not
    i = 0
    while frontier:
        i += 1
        if verbose:print("{}: {} states processed in {:.3f} seconds.".format(i,
			len(reachables), time.perf_counter()-clock))
        newfrontier = set()
       
        for (st_A, set_B, compr) in frontier:
            for sym in alph:
                new_A = dict_A[(st_A, sym)]
                if new_A == sink_A:
                    # We can forget the sink state of A
                    continue
                news_B = frozenset(st for sts in set_B for st in dict_B[(sts, sym)])
                if not news_B:
                    # A accepts but B does not: A\B is nonempty
                    if track:
                        return reachables[compr] + [sym]
                    else:
                        return True
                new_c = compre(new_A, news_B)
                if new_c in reachables:
                    continue
                if not track:
                    reachables.add(new_c)
                else:
                    reachables[new_c] = reachables[compr] + [sym]
                newfrontier.add((new_A, news_B, new_c))
        frontier = newfrontier
    return False

def verify_results():
    alph = [(0,0),(0,1),(1,0),(1,1)]
    clock = time.perf_counter()
    t4_dict, t4_states, t4_C = right_trace_aut(4)
    t4d_dict, t4d_states, t4d_init, t4d_finals = determinize(t4_dict, t4_states, alph)
    t4m_dict, t4m_states, t4m_init, t4m_finals = minimize(t4d_dict,
		t4d_states, t4d_init, t4d_finals, alph)
    t4m_sink = (t4m_states - t4m_finals).pop()
    print("Constructed minimal radius-4 trace automaton, took {:.3f} seconds.".format(
		time.perf_counter()-clock))
    print("For this trace, C is {}.".format(t4_C))
    clock = time.perf_counter()
    t6p_dict, t6p_states, _ = right_trace_aut(6, end_period=True)
    print("Constructed 3-periodic radius-6 trace automaton, took {:.3f} seconds.".format(
                time.perf_counter()-clock))
    clock = time.perf_counter()
    if diff_nonempty(t4m_dict, t4m_init, t4m_sink, t6p_dict, t6p_states, alph, verbose=True):
        print("Trace not periodizable at distance 6, took {:.3f} seconds.".format(
	            time.perf_counter()-clock))
    else:
        print("Trace is periodizable at distance 6, took {:.3f} seconds.".format(
	            time.perf_counter()-clock))

if __name__ == "__main__":
    verify_results()
\end{Verbatim}



\end{document}